\documentclass[12pt]{amsart}

\usepackage{hyperref}
\usepackage{verbatim}

\usepackage{macros}
\usepackage{amsthm}

\usepackage{tikz}
\usepackage[enableskew]{youngtab}
\usepackage[smalltableaux]{ytableau}
\Yboxdim{5pt}

\newtheorem{thm}{Theorem}[section]
\newtheorem{lem}[thm]{Lemma}
\newtheorem{prop}[thm]{Proposition}
\newtheorem{cor}[thm]{Corollary}

\theoremstyle{definition}

\newtheorem{defn}[thm]{Definition}

\newtheorem{ex}[thm]{Example}
\newtheorem{rmk}[thm]{Remark}

\definecolor{mypurp}{RGB}{153,50,204}

\makeatletter
\renewcommand{\boxed}[1]{\text{\fboxsep=.2em\fbox{\m@th$\displaystyle#1$}}}
\makeatother

\makeatletter
\let\c@thm\c@figure
\makeatother
\title{Cluster Duality for Lagrangian and Orthogonal Grassmannians}
\author{Charles Wang}
\address{Department of Mathematics, Harvard University, Cambridge, MA, USA}
\email{cmwang@math.harvard.edu}
\date{}

\begin{document}

\begin{abstract}
  In \cite{RW} Rietsch and Williams relate cluster structures and mirror symmetry for type A Grassmannians $\Gr(k,n)$, and use this interaction to construct Newton-Okounkov bodies and associated toric degenerations. In this article we define a cluster seed for the Lagrangian Grassmannian, and prove that the associated Newton-Okounkov body agrees up to unimodular equivalence with a polytope obtained from the superpotential defined by Pech and Rietsch on the mirror Orthogonal Grassmannian in \cite{pech-rietsch}. 
\end{abstract}

\maketitle

\tableofcontents
\section{Introduction}

In \cite{RW} Rietsch and Williams view open subsets of the Langlands dual Grassmannians $\Gr(n-k,n) \setminus D_{ac}\cong \Gr(k,n)\setminus D^\vee_{ac}$ in two ways: one as an $\cX$-cluster variety and the other as an $\cA$-cluster variety. Roughly speaking, cluster varieties are unions of algebraic tori $(\mC^*)^N$, indexed by combinatorial objects called \emph{seeds}, identified along certain \emph{mutation maps}. They then study combinatorial data on both sides: the Newton-Okounkov body associated to an $\cX$-torus and the superpotential polytope associated to a corresponding dual $\cA$-torus. Their main result is that these data agree, i.e. the polytopes are equal. Their strategy is to identify a particular seed for which they prove the equality ``by hand,'' and then to argue that the polytopes transform in the same way under seed mutations.

Let $\mX=\LGr(n,2n)$ be the variety of $n$-dimensional Lagrangian subspaces of $\mC^{2n}$ with respect to the symplectic form $\omega_{ij}=(-1)^j\delta_{i,2n+1-j}$. $\mX$ is a \emph{homogeneous space} of Dynkin type $C$, i.e. it can be written as $\Sp_n/P$ for a parabolic subgroup $P\subset \Sp_n$. We consider its embedding as a subvariety of $\Gr(n,2n)$ in its Pl\"ucker embedding $\Gr(n,2n)\hookrightarrow \mP(\wedge^n \mC^{2n})$. We will index Pl\"ucker coordinates on $\Gr(n,2n)$ by elements of $\binom{[2n]}{n}$, the set of $n$-subsets of $[2n]:=\{1,2,\dots, 2n\}$, or alternatively by Young diagrams $\lambda\subset n\times n$ fitting inside the $n\times n$ square. We translate between these two notations by the following bijection. Consider lattice paths in the $n\times n$ rectangle which start in the upper right corner and end in the lower left corner with unit steps (either down or to the left) labelled sequentially by $1,2,\dots, 2n$. Associate to $I\in\binom{[2n]}{n}$ the partition $\lambda_I$ lying above the lattice path whose vertical steps are labelled by the elements of $I$. For example, for $n=3$, the partition associated to $\{1,3,5\}$ is $\yng(3,2,1)$. In particular, Pl\"ucker coordinates for $\mX$ will be labelled interchangeably by $n$-subsets of $[2n]$ and Young diagrams contained in the $n\times n$ square. $\mX$ has dimension $N\coloneq \binom{n+1}{2}$, and a distinguished anticanonical divisor $D_{ac}=D_0+\dots+D_{n}$ made up of the $n+1$ hyperplanes $D_i=\{p_{n\times i}=0\}=\{p_{(n-i+1)\dots (2n-i)}\}$, where $n\times i$ denotes the corresponding Young diagram. In this article, $\mX$ takes the role of the $\cX$-cluster variety. 

Unlike the situation for the type $A$ Grassmannians, the Langlands dual Grassmannian $\mX^\vee$ is not isomorphic to $\mX$. Roughly speaking, we can associate to the Lie group $G=\Sp_n$ the data of the character lattice $\chi$ of a maximal torus $T$, and the root system $\Phi\subset\chi$. Then there is a unique Lie group $G^\vee$, called the \emph{Langlands dual group}, having as root system the coroots $\Phi^\vee$ and as character lattice the cocharacter lattice $\chi^\vee$. The parabolic subgroup $P\subset \Sp_n$ above then corresponds to some $P^\vee\subset G^\vee$, and we then set $\mX^\vee=P^\vee\backslash G^\vee$, and call this the \emph{Langlands dual Grassmannian}. 

For $\mX=\LGr(n,2n)$, $\mX^\vee$ is the orthogonal Grassmannian $\mathrm{OG}^{co}(n+1,2n+1)$ of co-isotropic $(n+1)$-dimensional subspaces of $\mC^{2n+1}$ with respect to a quadratic form $Q$. (It is isomorphic to the orthogonal Grassmannian $\mathrm{OG}(n,2n+1)$ of isotropic $n$-dimensional subspaces of $\mC^{2n+1}$ with respect to $Q$.) Following \cite{pech-rietsch}, we consider $\mX^\vee$ in its minimal embedding $\mX^\vee\hookrightarrow \mP(V^*)$, where $V$ is the irreducible representation corresponding to the parabolic subgroup $P^\vee$ ($P^\vee$ will be a maximal parabolic subgroup since $P$ was). As noted in \cite[\S3]{pech-rietsch}, because $\mX$ is \emph{cominuscule}, its cohomology is isomorphic (by the geometric Satake correspondence) to $V$. In this article, $\mX^\vee$ takes the role of the $\cA$-cluster variety. 

In this article, we carry out the first step of the \cite{RW} strategy for $\mX$ and $\mX^\vee$. We identify a particular seed, which we call the \emph{co-rectangles seed}, and show that the Newton-Okounkov body corresponding to this seed is unimodularly equivalent to the superpotential polytope defined using the Landau-Ginzburg model studied in \cite{pech-rietsch}. 

The outline of this article is as follows. In section 2, we define our co-rectangles seed. In section 3, we define the Newton-Okounkov body $\Delta_{\text{co-rect}}$ and the superpotential polytope $\Gamma$, and relate $\Gamma$ to a chain polytope. In section 4, we prove that $\Delta_{\text{co-rect}}$ and $\Gamma$ are unimodularly equivalent, and in section 5, we describe upcoming work. 

\subsection*{Acknowledgements}
The author is grateful to Konstanze Rietsch, Bernd Sturmfels, and Lauren Williams for many helpful discussions, comments, and suggestions.

\section{The co-rectangles Seed}

Seeds in a \emph{cluster structure} of rank $l$ for a commutative algebra are specified by a pair $({\bf x},B)$ of \emph{cluster variables} ${\bf x}=(x_1,\dots, x_m)$ and an $m\times l$ \emph{extended exchange matrix} $B$, for some $m\ge l$. If the topmost $l\times l$ square submatrix of $B$ is skew-symmetric, we can replace $B$ with a \emph{quiver} $Q$, which is a directed, oriented graph which may have parallel edges, but no $2$-cycles or loops, and some vertices designated as 'frozen.' For brevity, we do not give a full definition of a cluster algebra, and instead refer to \cite[\S3.1]{fwz1-3}.

The seeds for the coordinate rings of Grassmannians $\Gr(k,n)$ studied in \cite{RW}, which were first proven to give a cluster structure in \cite{scott}, can be described by quivers, and furthermore certain seeds admit an additional description in terms of certain planar, bicolored graphs called \emph{plabic graphs}. When we wish to distinguish the vertices of a plabic graph according to the bicoloring, we will refer to them as hollow ($\circ$) or filled ($\bullet$). Roughly speaking, ${\bf x}$ corresponds to the set of face labels of a plabic graph $G$, and $Q$ corresponds to the dual graph of $G$. A more thorough exposition of plabic graphs can be found in \cite{post}, where they were first introduced, and the relationship between cluster seeds and plabic graphs for Grassmannians can be found in \cite[\S5-6]{RW}. 

\subsection{The co-rectangles Symmetric Plabic Graph}
For the Lagrangian Grassmannian, the seed we are interested in studying is associated to an extended exchange matrix $B_{\text{co-rect}}$ whose top square submatrix is not skew-symmetric. This seems to be related to the fact that the type $C_n$ Dynkin diagram is not simply-laced. However, certain plabic seeds for the Grassmannian can be used to obtain seeds for the Lagrangian Grassmannian by \emph{quiver folding}. We first recall the analogous notion of \emph{symmetric plabic graphs}, due to \cite{karpman}.

\begin{defn}[{\cite[Def. 5.1]{karpman}}]
  A \emph{symmetric plabic graph} for $\mX$ is a plabic graph $G$ with $2n$ boundary vertices, labelled clockwise by $1,\dots, 2n$, and a distinguished diameter $d$ of the bounding disk satisfying the following conditions:
  \begin{enumerate}
  \item $d$ has one endpoint between vertices $2n$ and $1$, and the other between $n$ and $n+1$.
  \item No vertex of $G$ lies on $d$.
  \item Reflecting $G$ through $d$ gives a graph identical to $G$ with the colors of vertices reversed. 
  \end{enumerate}
\end{defn}

Our seed comes from the \emph{co-rectangles} symmetric plabic graph $G^{\text{co-rect}}_n$ (face are labelled by{\bf co}mplements of {\bf rect}angular Young diagrams in the $n\times n$ square). We define $G^{\text{co-rect}}_n$ by example for $n=4$ (see \ref{sym-plab-ex}), and give the associated dual quiver in \ref{quiver-ex} and folding in \ref{matrix-ex}. The extension to arbitrary $n$ is straightforward. We note that our $G^{\text{co-rect}}_n$ is mutation equivalent to $G^{\text{rec}}_{n,2n}$ of \cite{RW}, so in particular satisfies a technical assumption called \emph{reducedness}. 

  \begin{figure}[h!]\label{sym-plab-ex}
    \centering
    \begin{tikzpicture}
      \node (1) at (6,0) {$1$};
      \node (2) at (4,0) {$2$};
      \node (3) at (2,0) {$3$};
      \node (4) at (0,0) {$4$};
      \node (5) at (8,8) {$5$};
      \node (6) at (8,6) {$6$};
      \node (7) at (8,4) {$7$};
      \node (8) at (8,2) {$8$};
      
      \node (9) at (7,6) {$\circ$};
      \node (10) at (5,6) {$\circ$};
      \node (11) at (3,6) {$\circ$};
      \node (12) at (6,5) {\textbullet};
      \node (13) at (4,5) {\textbullet};
      \node (14) at (2,5) {\textbullet};
      \node (15) at (7,4) {$\circ$};
      \node (16) at (5,4) {$\circ$};
      \node (17) at (3,4) {$\circ$};
      \node (18) at (6,3) {\textbullet};
      \node (19) at (4,3) {\textbullet};
      \node (20) at (2,3) {\textbullet};
      \node (21) at (7,2) {$\circ$};
      \node (22) at (5,2) {$\circ$};
      \node (23) at (3,2) {$\circ$};
      \node (24) at (6,1) {\textbullet};
      \node (25) at (4,1) {\textbullet};
      \node (26) at (2,1) {\textbullet};

      \node (27) at (0,3) {$\circ$};
      \node (28) at (5,8) {\textbullet};

      \draw (28)--(9)--(12)--(15)--(18)--(21)--(24)--(1);
      \draw (28)--(10)--(13)--(16)--(19)--(22)--(25)--(2);
      \draw (28)--(11)--(14)--(17)--(20)--(23)--(26)--(3);
      \draw (5)--(28) to [out=180,in=90] (27);
      \draw (27)--(4);

      \draw (27)--(14)--(11)--(13)--(10)--(12)--(9)--(6);
      \draw (27)--(20)--(17)--(19)--(16)--(18)--(15)--(7);
      \draw (27)--(26)--(23)--(25)--(22)--(24)--(21)--(8);

      \node (29) at (1,1) {\yng(4,4,4)};
      \node (30) at (3,1) {\yng(4,4)};
      \node (31) at (5,1) {\yng(4)};
      \node (32) at (7,1) {$\varnothing$};
      \node (33) at (2,2) {\yng(4,4,4,1)};
      \node (34) at (3.5,2.5) {\yng(4,4,1,1)};
      \node (35) at (5.5,2.5) {\yng(4,1,1,1)};
      \node (36) at (7.5,2.5) {\yng(1,1,1,1)};
      \node (37) at (2,4) {\yng(4,4,4,2)};
      \node (38) at (3.5,4.5) {\yng(4,4,2,2)};
      \node (39) at (5.5,4.5) {\yng(4,2,2,2)};
      \node (40) at (7.5,4.5) {\yng(2,2,2,2)};
      \node (41) at (2,6) {\yng(4,4,4,3)};
      \node (42) at (4,6) {\yng(4,4,3,3)};
      \node (43) at (6,6) {\yng(4,3,3,3)};
      \node (44) at (7,7) {\yng(3,3,3,3)};
      \node (45) at (1,7) {\yng(4,4,4,4)};
      \node (46) at (-2,2) {};
      \node (47) at (0,8) {};
      \node (48) at (6,10) {};
      
      \draw (1)--(2)--(3)--(4);
      \draw (4) to [out=180,in=270] (46) to [out=90,in=225] (47) to [out=45,in=180] (48) to [out=0,in=90] (5);
      \draw (5)--(6)--(7)--(8);
      \draw (8) to [out=270,in=0] (1);

      \node (49) at (7.5,.5) {};
      \draw[dashed] (49)--(47);
    \end{tikzpicture}
    \caption{The co-rectangles symmetric plabic graph $G^{\text{co-rect}}_4$.}
  \end{figure}
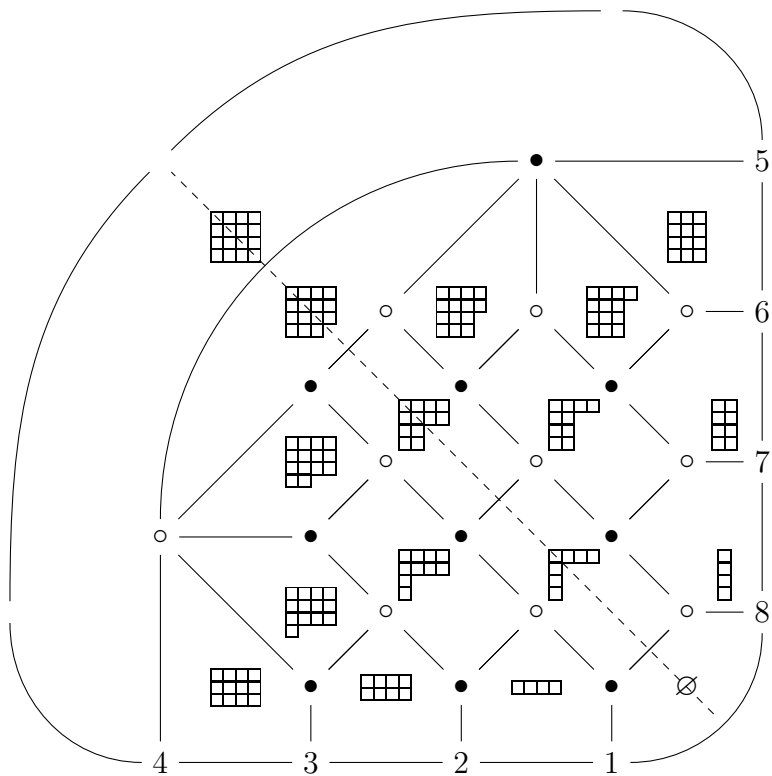

\begin{rmk}
  The face labels in \ref{sym-plab-ex} are auxiliary data associated to the graph, and the procedure to obtain them is described in \cite[Definition 3.5]{RW}.
\end{rmk}

Next we give the dual quiver, which corresponds to an $\cX$-seed for $\Gr(4,8)$.

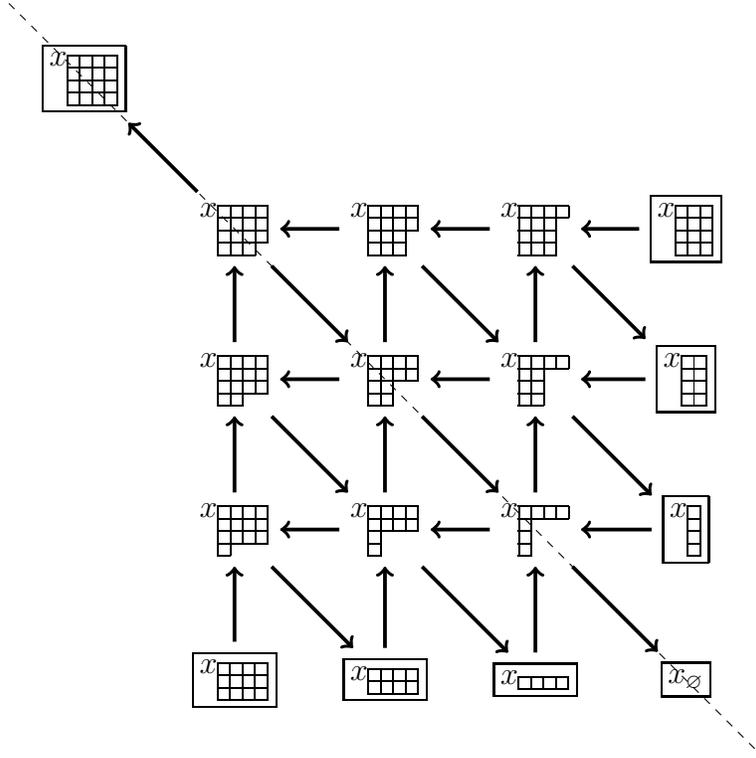
\begin{figure}[h!] \label{quiver-ex}
  \begin{tikzpicture}

    \node(0) at (0,8) {$\boxed{x_{\yng(4,4,4,4)}}$};
    \node(1) at (2,6) {$x_{\yng(4,4,4,3)}$};
    \node(2) at (4,4) {$x_{\yng(4,4,2,2)}$};
    \node(3) at (6,2) {$x_{\yng(4,1,1,1)}$};
    \node(4) at (8,0) {$\boxed{x_{\varnothing}}$};
    \node(5) at (2,4) {$x_{\yng(4,4,4,2)}$};
    \node(6) at (2,2) {$x_{\yng(4,4,4,1)}$};
    \node(7) at (2,0) {$\boxed{x_{\yng(4,4,4)}}$};
    \node(8) at (4,6) {$x_{\yng(4,4,3,3)}$};
    \node(9) at (4,2) {$x_{\yng(4,4,1,1)}$};
    \node(10) at (4,0) {$\boxed{x_{\yng(4,4)}}$};
    \node(11) at (6,6) {$x_{\yng(4,3,3,3)}$};
    \node(12) at (6,4) {$x_{\yng(4,2,2,2)}$};
    \node(13) at (6,0) {$\boxed{x_{\yng(4)}}$};
    \node(14) at (8,6) {$\boxed{x_{\yng(3,3,3,3)}}$};
    \node(15) at (8,4) {$\boxed{x_{\yng(2,2,2,2)}}$};
    \node(16) at (8,2) {$\boxed{x_{\yng(1,1,1,1)}}$};

    \draw[->, line width=.5mm] (1) edge (0);
    \draw[->, line width=.5mm] (1) edge (2);
    \draw[->, line width=.5mm] (2) edge (3);
    \draw[->, line width=.5mm] (3) edge (4);
    \draw[->, line width=.5mm] (8) edge (1);
    \draw[->, line width=.5mm] (11) edge (8);
    \draw[->, line width=.5mm] (14) edge (11);
    \draw[->, line width=.5mm] (5) edge (1);
    \draw[->, line width=.5mm] (6) edge (5);
    \draw[->, line width=.5mm] (7) edge (6);
    \draw[->, line width=.5mm] (10) edge (9);
    \draw[->, line width=.5mm] (9) edge (2);
    \draw[->, line width=.5mm] (2) edge (8);
    \draw[->, line width=.5mm] (13) edge (3);
    \draw[->, line width=.5mm] (3) edge (12);
    \draw[->, line width=.5mm] (12) edge (11);
    \draw[->, line width=.5mm] (15) edge (12);
    \draw[->, line width=.5mm] (12) edge (2);
    \draw[->, line width=.5mm] (2) edge (5);
    \draw[->, line width=.5mm] (16) edge (3);
    \draw[->, line width=.5mm] (3) edge (9);
    \draw[->, line width=.5mm] (9) edge (6);
    \draw[->, line width=.5mm] (5) edge (9);
    \draw[->, line width=.5mm] (9) edge (13);
    \draw[->, line width=.5mm] (6) edge (10);
    \draw[->, line width=.5mm] (8) edge (12);
    \draw[->, line width=.5mm] (12) edge (16);
    \draw[->, line width=.5mm] (11) edge (15);

    \draw[dashed] (-1,9) -- (9,-1);
    
  \end{tikzpicture}
  \caption{The dual quiver $Q$ to $G^{\text{co-rect}}_4$. Vertices of $Q$ are labelled by the $\cX$-cluster variables. Edges of $Q$ are directed such that when crossing an edge of the plabic graph, the hollow vertex is to the left. Boxes have been drawn around frozen vertices, which correspond to faces of the plabic graph adjacent to the boundary disk.}
\end{figure}

The associated exchange matrix is the $17\times 9$ matrix $B$ whose rows are indexed by any of the Young diagrams in \ref{quiver-ex} and whose columns are indexed by the non-boxed Young diagrams. The $(\mu,\nu)$ entry of $B$ is given by:
\[
  B_{\mu,\nu}=
  \begin{cases}
    1 & \mu\rightarrow\nu \textrm{ is an edge in }Q\\
    -1 & \mu\leftarrow\nu \textrm{ is an edge in }Q\\
    0 & \textrm{otherwise} 
  \end{cases}
\]

Finally, to obtain an $\cX$-seed for $\mX$, we fold the quiver \ref{quiver-ex} by the involution induced by sending a vertex to its reflection about the dashed diagonal. Because we will not make use of it in this article, we do not define quiver folding, and we refer to \cite[\S4.4]{fwz4-5} for the full definition. We instead just give the $11\times 6$ extended exchange matrix $B^{\text{co-rect}}_4$ of the folding:

\begin{ex}\label{matrix-ex}
  The columns of the folded matrix are indexed (in order) by the mutable orbits of the involution. These are: $\{\yng(4,4,4,3)\}, \{\yng(4,4,2,2)\}, \{\yng(4,1,1,1)\}, \{\yng(4,4,4,2), \yng(4,4,3,3)\}, \{\yng(4,4,4,1),\yng(4,3,3,3)\}, \{\yng(4,4,1,1),\yng(4,2,2,2)\}$. The first six rows will be indexed in the same order. The remaining 5 rows will be indexed (in order) by the frozen orbits of the involution $\{\yng(4,4,4,4)\}, \{\yng(4,4,4),\yng(3,3,3,3)\}, \{\yng(4,4),\yng(2,2,2,2)\}, \{\yng(4),\yng(1,1,1,1)\}, \{\varnothing\}$. 

  \[B^{\text{co-rect}}_4=
    \left(
      \begin{array}{cccccc}
        0&1&0&-1&0&0\\
        -1&0&1&1&0&-1\\
        0&-1&0&0&0&1\\
        2&-2&0&0&-1&1\\
        0&0&0&1&0&-1\\
        0&2&-2&-1&1&0\\
        -1&0&0&0&0&0\\
        0&0&0&0&1&0\\
        0&0&0&0&-1&1\\
        0&0&2&0&0&-1\\
        0&0&-1&0&0&0\\
      \end{array}
    \right)
  \]
  \end{ex}

\subsection{The Network Parametrization ($\cX$-cluster seed) for $\mX$}
We will now describe how to use a symmetric plabic graph to construct a \emph{network torus} in $\mX$. This will allow us to compute valuations associated to a seed using plabic graphs in the following section. We summarize the presentation in \cite[\S6]{RW}.

\begin{defn}
  A \emph{perfect orientation} $O$ of a plabic graph $G$ is an orientation of each edge of $G$ such that each filled internal vertex is incident to exactly one edge directed away from it, and each hollow vertex is incident to exactly one edge directed towards it. The source set $I_O$ of $O$ is the set of boundary vertex labels which are sources of $G$ as a directed graph with edge directions $O$. 
\end{defn}

Let $G$ denote a plabic graph with a perfect orientation $O$. If we need further assumptions on $G$, they will be stated explicitly. 

\begin{defn}
  Let $J$ be a subset of the boundary vertices of $G$ with $|J|=|I_O|$. A \emph{flow} from $I_O$ to $J$ is a collection of pairwise vertex-disjoint paths with sources $I_O\setminus (I_O\cap J)$ and sinks $J\setminus (I_O\cap J)$. 
\end{defn}

Because each path $p$ in a flow $F$ begins and ends at a boundary vertex of $G$, $p$ partitions the faces of $G$ into two sets, those to the left of $p$ and those to the right of $p$ in the direction of the path. Let $p_L$ denote the set of face labels to the left of $p$.

\begin{defn}
  For a path $p$ in a flow $F$, we define the \emph{weight} of $p$ to be $\mathrm{wt}(p)=\prod_{\lambda \in p_L} x_\lambda$. For a flow $F$, we define the \emph{weight} to be $\mathrm{wt}(F)=\prod_{p\in F}\mathrm{wt}(p)$. Finally, for a subset $J$ of the boundary vertices of $G$ with $|J|=|I_O|$, let $\cF$ denote the set of all flows from $I_O$ to $J$, and define the \emph{flow polynomial} $P_J^G=\sum_{F\in\cF}\mathrm{wt}(F)$. 
\end{defn}

Now let $G=G^{\text{co-rect}}_n$ be the co-rectangles plabic graph. In what follows, we will need a perfect orientation $O_{\text{co-rect}}$ on $G$, defined as follows.

\begin{defn}
  Set $\{1,2,\dots, n\}$ to be sources, and $\{n+1,n+2,\dots, 2n\}$ to be sinks. (The edges adjacent to vertices $1\le i\le n$ will be directed away from $i$, and the edges adjacent to vertices $n+1\le i\le 2n$ will be directed towards $i$.) Because symmetric plabic graphs are also usual plabic graphs, then there is a unique such perfect orientation by \cite[Lemma 4.5]{psw}, see \cite[Remark 6.4]{RW}. We call this $O_{\text{co-rect}}$.
\end{defn}
This is the choice of perfect orientation we will use for the rest of the article. For an example of the above definitions, see \ref{graph-stuff}, where we give our perfect orientation for $n=3$, and compute a flow polynomial.

Next, let $S=\{x_{\mu}\mid \mu\textrm{ is a face label of }G^{\text{co-rect}}_n\}$ be the set of face labels of the co-rectangles plabic graph. We think of these as coordinates on the \emph{network torus} $\mathbb{T}_G\cong (\mC^*)^{|S|}$, and we use the flow polynomials to define an embedding of $\mathbb{T}_G$ into $\Gr(n,2n)$.

\begin{thm}[{{\cite[Theorem 12.7]{post}}},{{\cite[Theorem 6.8]{RW}}}]
  Let $G$ be the co-rectangles plabic graph, and $J\in\binom{[2n]}{n}$. Consider the map $\Phi:\mathbb{T}_G\rightarrow \Gr(n,2n)$ defined by sending $(x_\mu\mid\mu\in S)\in \mathbb{T}_G\mapsto (P_J^G(x_\mu)\mid J\in\binom{[2n]}{n})\in \Gr(n,2n)$. Then $\Phi$ is well-defined, and gives an embedding $\mathbb{T}_G\hookrightarrow\Gr(n,2n)$. 
\end{thm}

Finally, let $G=G^{\text{co-rect}}_n$ be the co-rectangles symmetric plabic graph. We define the equivalence relation $\sim$ on $S$ given by $x_{\mu}\sim x_{\mu^T}$, where $\mu^T$ denotes the transpose partition to $\mu$. In Karpman's language, this corresponds to taking a \emph{symmetric weighting}, and Karpman shows that restricting to these weightings gives an embedding whose image lands inside of $\mX\subset \Gr(n,2n)$. We think of $S/\sim$ as coordinates on the network torus $\mathbb{T}_{\text{co-rect}}\cong (\mC^*)^{|S/\sim|}$, and we use the flow polynomials to define an embedding of $\mathbb{T}_{\text{co-rect}}$ into $\mX$. 

\begin{thm}[{{\cite[Theorem 5.15]{karpman}}}]
  Let $G=G^{\text{co-rect}}_n$ be the co-rectangles symmetric plabic graph, and $J\in \binom{[2n]}{n}$. Consider the map $\Phi:\mathbb{T}_{\text{co-rect}}\hookrightarrow \mX$ which is defined by sending $(x_\mu\mid \mu\in S/\sim)\in \mathbb{T}_{\text{co-rect}}\mapsto (P_J^G(x_\mu=x_{\mu^T})\mid J\in \binom{[2n]}{n})\in \mX$. Then $\Phi$ is well-defined, and gives an embedding $\mathbb{T}_{\text{co-rect}}\hookrightarrow \mX$. 
\end{thm}

\begin{rmk}
  Although Karpman's paper is written in terms of edge weightings, the translation to face weightings can be found in \cite[Lemma 11.2]{post}. 
\end{rmk}

Thus we associate to $G^{\text{co-rect}}_n$ and $O_{\text{co-rect}}$ a dense torus $\mathbb{T}_{\text{co-rect}}\hookrightarrow\mX$. On the level of coordinate rings, this induces an injection $\mC[\mX]\hookrightarrow \mC[\mathbb{T}_{\text{co-rect}}]$, so we may express polynomials in the Pl\"ucker coordinates on $\mX$ as Laurent polynomials in the coordinates on $\mathbb{T}_{\text{co-rect}}$.

\section{Polytopes}\label{polytopes}

\subsection{The Newton-Okounkov body $\Delta_{\text{co-rect}}$}
We associate to the co-rectangles symmetric plabic graph $G^{\text{co-rect}}_n$ and ample divisor $D=D_n$ a Newton Okounkov body $\Delta_{\text{co-rect}}(D)$ by the following procedure, following \cite[Definition 8.1]{RW}. First, we define a valuation using the inclusion $\mC[\mX]\hookrightarrow \mC[\mathbb{T}_{\text{co-rect}}]$ obtained at the end of the previous section. 

\begin{defn}\label{val-corect}
  Fix a total order on the torus coordinates $S$ defined at the end of the previous section. Then we define the valuation $\val_{\text{co-rect}}:\mC[\mX]\setminus \{0\}\rightarrow \mZ^{|S|}$ by sending $f\in \mC[\mX]$ to the exponent vector of the lexicographically minimal term when $f$ is viewed as an element of $\mC[\mathbb{T}_{\text{co-rect}}]$, i.e. as a Laurent polynomial in the torus coordinates. 
\end{defn}

Now, using this valuation, we define the Newton-Okounkov body:

\begin{defn}
  Let $\val_{\text{co-rect}}$ be as above. Then we define 
  \[\Delta_{\text{co-rect}}=\overline{\conv\left(\bigcup_{r=1}^\infty \frac{1}{r} \val_{\text{co-rect}}(H^0(\mX,\cO(rD)))\right)}\]
\end{defn}

Concretely, the nonzero sections in $H^0(\mX,\cO(rD))$ can be identified with Laurent polynomials whose numerators are degree $r$ homogeneous polynomials in the Pl\"ucker coordinates of $\mX$, and whose denominators are the Pl\"ucker coordinate $p_{n\times n}^r$. For $G^{\text{co-rect}}_n$ and $O_{\text{co-rect}}$, the only flow from $[n]$ to $[n]$ is the empty flow, so the expression of $p_{n\times n}$ on the torus $\mathbb{T}_{\text{co-rect}}$ is $1$. Therefore, computing valuations of sections $H^0(\mX,\cO(rD))$ reduces to computing valuations of elements of $\mC[\mX]$, so we can use the valuation \ref{val-corect}. 

\begin{rmk}
Although the valuation $\val_{\text{co-rect}}$ depended upon a choice of total order of the torus coordinates $S$, the Newton-Okounkov body $\Delta_{\text{co-rect}}$ does not, and we will not make use of any choice of total order in our proofs.
\end{rmk}

\begin{ex}\label{graph-stuff}
  For $\LGr(3,6)$, which will be our running example, we give our co-rectangles plabic graph as well as a perfect orientation. We also give two flows.

  \begin{figure}[h!]\label{top-flow}
    \begin{tikzpicture}
      \node (1) at (4,0) {\color{mypurp} $1$};
      \node (2) at (2,0) {\color{mypurp} $2$};
      \node (3) at (0,0) {\color{mypurp} $3$};
      \node (4) at (6,6) {\color{mypurp} $4$};
      \node (5) at (6,4) {\color{mypurp} $5$};
      \node (6) at (6,2) {$6$};
      
      \node (7) at (4,1) {$\bullet$};
      \node (8) at (5,2) {$\circ$};
      \node (9) at (4,3) {$\bullet$};
      \node (10) at (5,4) {$\circ$};
      \node (11) at (3,6) {$\bullet$};
      \node (12) at (0,3) {$\circ$};
      
      \node (13) at (2,1) {$\bullet$};
      \node (14) at (3,2) {$\circ$};
      \node (15) at (2,3) {$\bullet$};
      \node (16) at (3,4) {$\circ$};
      
      \draw[->, line width=.5mm] (1) edge (7);
      \draw[->, line width=.5mm] (7) edge (8);
      \draw[->, line width=.5mm] (8) edge (6);
      \draw[->, line width=.5mm] (8) edge (9);
      \draw[->, line width=.5mm, mypurp] (9) edge (10);
      \draw[->, line width=.5mm] (10) edge (11);
      \draw[->, line width=.5mm, mypurp] (11) edge (4);
      \draw[->, line width=.5mm, mypurp] (10) edge (5);
      \draw[->, line width=.5mm, mypurp] (2) edge (13);
      \draw[->, line width=.5mm, mypurp] (13) edge (14);
      \draw[->, line width=.5mm] (14) edge (7);
      \draw[->, line width=.5mm, mypurp] (14) edge (15);
      \draw[->, line width=.5mm, mypurp] (15) edge (16);
      \draw[->, line width=.5mm, mypurp] (16) edge (9);
      \draw[->, line width=.5mm, mypurp] (3) edge (12);
      \draw[->, line width=.5mm] (12) edge (13);
      \draw[->, line width=.5mm] (12) edge (15);
      \draw[->, line width=.5mm, mypurp] (12) edge (11);
      \draw[->, line width=.5mm] (11) edge (16);
      
      \node (17) at (1,5) {$\yng(3,3,3)$};          
      \node (18) at (2,4) {$\yng(3,3,2)$};
      \node (19) at (3.5,2.5) {$\yng(3,1,1)$};
      \node (20) at (5,1) {$\varnothing$};
      \node (21) at (3,1) {$\yng(3)$};
      \node (22) at (1,1) {$\yng(3,3)$};
      \node (23) at (5,3) {$\yng(1,1,1)$};
      \node (24) at (5,5) {$\yng(2,2,2)$};
      \node (25) at (2,2) {$\yng(3,3,1)$};
      \node (26) at (4,4) {$\yng(3,2,2)$};
    \end{tikzpicture}
    \caption{The co-rectangles symmetric plabic graph for $n=3$, with acyclic perfect orientation, and (minimal) flow from $\{1,2,3\}$ to $\{1,5,4\}$ in purple.}
  \end{figure}
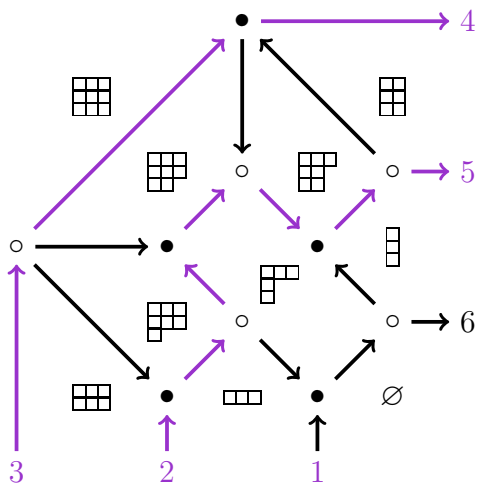
  
  \begin{figure}[h!]\label{bottom-flow}
    \begin{tikzpicture}
      \node (1) at (4,0) {\color{mypurp} $1$};
      \node (2) at (2,0) {\color{mypurp} $2$};
      \node (3) at (0,0) {\color{mypurp} $3$};
      \node (4) at (6,6) {\color{mypurp} $4$};
      \node (5) at (6,4) {\color{mypurp} $5$};
      \node (6) at (6,2) {$6$};
      
      \node (7) at (4,1) {$\bullet$};
      \node (8) at (5,2) {$\circ$};
      \node (9) at (4,3) {$\bullet$};
      \node (10) at (5,4) {$\circ$};
      \node (11) at (3,6) {$\bullet$};
      \node (12) at (0,3) {$\circ$};
      
      \node (13) at (2,1) {$\bullet$};
      \node (14) at (3,2) {$\circ$};
      \node (15) at (2,3) {$\bullet$};
      \node (16) at (3,4) {$\circ$};
      
      \draw[->, line width=.5mm] (1) edge (7);
      \draw[->, line width=.5mm, mypurp] (7) edge (8);
      \draw[->, line width=.5mm] (8) edge (6);
      \draw[->, line width=.5mm, mypurp] (8) edge (9);
      \draw[->, line width=.5mm, mypurp] (9) edge (10);
      \draw[->, line width=.5mm] (10) edge (11);
      \draw[->, line width=.5mm, mypurp] (11) edge (4);
      \draw[->, line width=.5mm, mypurp] (10) edge (5);
      \draw[->, line width=.5mm, mypurp] (2) edge (13);
      \draw[->, line width=.5mm, mypurp] (13) edge (14);
      \draw[->, line width=.5mm, mypurp] (14) edge (7);
      \draw[->, line width=.5mm] (14) edge (15);
      \draw[->, line width=.5mm] (15) edge (16);
      \draw[->, line width=.5mm] (16) edge (9);
      \draw[->, line width=.5mm, mypurp] (3) edge (12);
      \draw[->, line width=.5mm] (12) edge (13);
      \draw[->, line width=.5mm] (12) edge (15);
      \draw[->, line width=.5mm, mypurp] (12) edge (11);
      \draw[->, line width=.5mm] (11) edge (16);
      
      \node (17) at (1,5) {$\yng(3,3,3)$};          
      \node (18) at (2,4) {$\yng(3,3,2)$};
      \node (19) at (3.5,2.5) {$\yng(3,1,1)$};
      \node (20) at (5,1) {$\varnothing$};
      \node (21) at (3,1) {$\yng(3)$};
      \node (22) at (1,1) {$\yng(3,3)$};
      \node (23) at (5,3) {$\yng(1,1,1)$};
      \node (24) at (5,5) {$\yng(2,2,2)$};
      \node (25) at (2,2) {$\yng(3,3,1)$};
      \node (26) at (4,4) {$\yng(3,2,2)$};
    \end{tikzpicture}
    \caption{A second flow from \{1,2,3\} to \{1,4,5\} for $n=3$.}
  \end{figure}
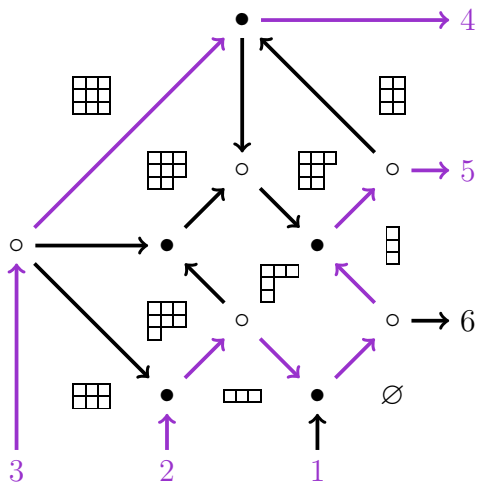

  For the top flow above \ref{top-flow}, there are no face labels to the left of the path $1\rightarrow 1$. The face labels to the left of $3\rightarrow 4$ are $\yng(3,3,3)$. The face labels to the left of $2\rightarrow 5$ are $\yng(3,3)$, $\yng(3,3,1)$, $\yng(3,3,2)$, $\yng(3,2,2)$, $\yng(2,2,2)$, and $\yng(3,3,3)$, contributing a monomial $x_{(3,3,3)}^2x_{(3,3)}^2x_{(3,3,1)}x_{(3,3,2)}x_{(3,2,2)}^2$.

  For the bottom flow above \ref{bottom-flow}, there are no face labels to the left of the path $1\rightarrow 1$. The face labels to the left of $3\rightarrow 4$ are $\yng(3,3,3)$. The face labels to the left of $2\rightarrow 5$ are $\yng(3,3), \yng(3,3,1), \yng(3,1,1), \yng(3,2,2), \yng(2,2,2), \yng(3,3,2), \yng(3,3,3)$, contributing a monomial $x_{(3,3,3)}^2x_{(3,3)}^2x_{(3,3,1)}x_{(3,3,2)}x_{(3,2,2)}^2x_{(3,1,1)}$

  These are the only flows from $\{1,2,3\}$ to $\{1,4,5\}$ for $G=G^{\text{co-rect}}_n$ and $O_{\text{co-rect}}$, so the flow polynomial is the sum of these
  \[P_{\{1,4,5\}}^G=(x_{(3,3,3)}^2x_{(3,3)}^2x_{(3,3,1)}x_{(3,3,2)}x_{(3,2,2)}^2)(1+x_{(3,1,1)})\]
    The minimal term is $(x_{(3,3,3)}^2x_{(3,3)}^2x_{(3,3,1)}x_{(3,3,2)}x_{(3,2,2)}^2)$, so the valuation is $(0,2,0,2,1,2)$, agreeing with the coordinates given in \ref{no-body} below. 
\end{ex}

Alternatively, because symmetric plabic graphs are also plabic graphs in the usual sense, we can compute Pl\"ucker coordinate valuations for plabic seeds more directly from Young diagrams.

\begin{defn}
  For any skew partition $\nu\subset n\times n$, we define $\mathrm{maxdiag}(\nu)$ to be the maximum number of boxes along any diagonal of slope $-1$.   
\end{defn}

\begin{prop} For $\mu\subset n\times n$ a face label of $G$ and $\lambda\subset n\times n$ arbitrary , we have
\[\val_{\text{co-rect}}(p_\lambda)_{\mu}=\begin{cases}\mathrm{maxdiag}(\mu\backslash\lambda)+\mathrm{maxdiag}(\mu^T\backslash\lambda) & \mu\neq\mu^T\\
    \mathrm{maxdiag}(\mu\backslash\lambda) & \mu=\mu^T\end{cases}\]
\end{prop}
  
\begin{proof}
  Because symmetric plabic graphs are also plabic graphs in the usual sense, then by \cite[Lemma 6.3]{RW} we may choose a perfect orientation $O$ with source set $\{1,2,\dots, n\}$. Recall that we think of $\lambda$ also as an $n$-subset of $[2n]$. Then, because flows of the symmetric plabic graph are the same as flows in the underlying plabic graph, there is a minimal flow $F$ from $[n]$ to $\lambda$ by \cite[Corollary 12.4]{RW}, and $\val_{\text{co-rect}}(p_\lambda)$ is the number of paths in $F$ which have the face labelled by $\mu$ to the left, plus the number of paths in $F$ which have the face labelled by $\mu^T$ to the left if $\mu\neq \mu^T$. Finally, by \cite[Corollary 16.19]{RW}, this is equal to $\mathrm{maxdiag}(\mu\backslash\lambda)$ if $\mu=\mu^T$, and $\mathrm{maxdiag}(\mu\backslash\lambda)+\mathrm{maxdiag}(\mu^T\backslash\lambda)$ if $\mu\neq\mu^T$. 
\end{proof}

\begin{ex}\label{no-body}
  For $\LGr(3,6)$, we have 14 Pl\"ucker coordinates with their valuations in coordinates $(156,126,145,125,124,123)$ (or in Young diagrams $({\tiny\yng(3)}, {\tiny\yng(3,3)}, {\tiny\yng(3,1,1)}, {\tiny\yng(3,3,1)}, {\tiny\yng(3,3,2)}, {\tiny\yng(3,3,3)})$):
  
  \[\begin{array}{c|c}
      I\in \binom{[6]}{3} & \val_{\text{co-rect}}(p_I)\\ \hline
      123 & (0, 0, 0, 0, 0, 0)\\
      124 & (0, 0, 0, 0, 0, 1)\\
      125=134 & (0, 1, 0, 1, 1, 1)\\
      126=234 & (1, 1, 1, 2, 1, 1)\\
      135 & (0, 2, 0, 2, 1, 1)\\
      136=235 & (1, 2, 1, 2, 1, 1)\\
      145 & (0, 2, 0, 2, 1, 2)\\
    \end{array}
    \quad
    \begin{array}{c|c}
      I\in \binom{[6]}{3} & \val_{\text{co-rect}}(p_I)\\ \hline
      146=245 & (1, 2, 1, 2, 1, 2)\\
      156=345 & (1, 3, 1, 3, 2, 2)\\
      236 & (2, 2, 1, 2, 1, 1)\\
      246 & (2, 2, 1, 2, 1, 2)\\
      256=346 & (2, 3, 1, 3, 2, 2)\\
      356 & (2, 4, 1, 4, 2, 2)\\
      456 & (2, 4, 1, 4, 2, 3)
    \end{array}
  \]
  and the inequalities defining the convex hull of these coordinates:
  \[\left(
      \begin{array}{ccccccc}
        0&0&-1&0&1&0&0\\
        0&-1&1&2&-1&0&0\\
        0&1&0&-1&0&0&0\\
        0&0&0&0&-1&2&0\\
        0&0&0&-1&1&-1&0\\
        0&0&0&0&0&-1&1\\  
        1&0&0&-1&0&0&0\\  
        1&1&0&-1&-1&1&0\\
        1&0&1&1&-1&-1&0\\
        1&0&1&0&0&-1&-1  
      \end{array}
    \right)
    \left(
      \begin{array}{c}
        1\\
        p_{156}\\
        p_{126}\\
        p_{145}\\
        p_{125}\\
        p_{124}\\
        p_{123}
      \end{array}
    \right)\ge 0
  \]

  The convex hull of these valuations has $f$-vector $(14,51,86,78,39,10)$ and volume $16=\deg\LGr(3,6)$. Although $456$ appears as a face label of the co-rectangles plabic graph, there is no flow beginning at $\{1,2,3\}$ with this face to the left, so every $\val_{\text{co-rect}}(p_\lambda)_{456}=0$ for any $\lambda$. Thus we exclude this coordinate in order to work with a full-dimensional polytope. 
\end{ex}

The fact that the volume of the convex hull of the valuations of the Pl\"ucker coordinates is equal to the degree of $\LGr(3,6)$ in the example above is not an accident. In fact, as we will see in the proof of \ref{main-thm}:

\begin{thm}
  $\conv(\{\val_{\text{co-rect}}(p_\lambda)\mid \lambda\subset n\times n\})=\Delta_{\text{co-rect}}$ is a Newton-Okounkov body for $\mX$ with respect to the valuation $\val_{\text{co-rect}}$. (Equivalently, the Pl\"ucker coordinates form a Khovanskii basis for $\mC[\mX]$ with respect to the valuation $\val_{\text{co-rect}}$.)
\end{thm}

\subsection{The superpotential polytope $\Gamma$}
We use the Laurent polynomial expression for the restriction of the superpotential $W_q$ to a torus $(\mC^*)^{\binom{n+1}{2}}\hookrightarrow \mX^\vee$ for the Landau-Ginzburg model for $\mX$ found by Pech and Rietsch:
\begin{defn}[{{\cite[Prop. A.1]{pech-rietsch}}}]
  Let coordinates on the torus above be given by $a_{ij}$ for $1\le i\le j\le n$, and let $\Lambda$ denote the set of strict partitions with at least one part of size $n$ that are contained in the maximal, right-justified staircase in the $n\times n$ square. For any $\lambda\in\Lambda$, label each box by $(i,j)$ where $i$ indexes the row and $j$ the column. Then set $\lambda_j$ to be the largest index such that $(\lambda_j,j)\in\lambda$. Then the restriction of the superpotential to this torus is given by
\[W_q=\sum_{i\le j\in [n]}a_{ij}+\sum_{\lambda \in \Lambda}\frac{q}{\prod_{j\in [n]}a_{\lambda_jj}}\]
\end{defn}

\begin{ex}\label{potential-ex}
  For $n=3$, the superpotential has $\binom{3+1}{2}+2^{3-1}=10$ terms:
  \[a_{11}+a_{12}+a_{13}+a_{22}+a_{23}+a_{33}+\frac{q}{a_{11}a_{12}a_{13}}+\frac{q}{a_{11}a_{12}a_{23}}+\frac{q}{a_{11}a_{22}a_{23}}+\frac{q}{a_{11}a_{22}a_{33}}\]
  where the last four terms correspond to the diagrams
  \ytableaushort{{11}{12}{13}},
  \ytableaushort{{11}{12}{13},\none\none{23}},
  \ytableaushort{{11}{12}{13},\none{22}{23}}, and
  \ytableaushort{{11}{12}{13},\none{22}{23},\none\none{33}}
\end{ex}

In order to define the superpotential polytope $\Gamma$, we first define tropicalization for Laurent polynomial whose coefficients are all positive, real numbers.

\begin{defn}[{{\cite[Def. 10.7]{RW}}}]
  For any Laurent polynomial $h$ in variables $z_1,\dots, z_k$ with coefficients in $\mR_{>0}$, we define $\Trop(h):\mR^k\rightarrow \mR$ inductively as follows. First, we set $\Trop(z_i)(y_1,\dots, y_k)=y_i$, and we denote this tropicalization by a capital letter $\Trop(z_i)=Z_i$. Next, if $h_1$ and $h_2$ are any Laurent polynomials with positive coefficients, and $c_1,c_2$ are any positive real numbers, then
  \[\Trop(c_1h_1+c_2h_2) = \min(\Trop(h_1),\Trop(h_2))\quad\text{and} \quad \Trop(h_1h_2)=\Trop(h_1)+\Trop(h_2) \]
This inductively defines $\Trop(h)$. 
\end{defn}

Following \cite[Def. 10.14]{RW}, we make the following definition for the $\Gamma$.

\begin{defn}
  Consider $W_q:\mR^{\binom{n}{2}}\times \mR\rightarrow \mR$ as a Laurent polynomial with positive coefficients in the variables $a_{ij}$ (corresponding to the first factor of $\mR^{\binom{n}{2}}$ and $q$ corresponding to the second factor of $\mR$. Then the \emph{superpotential polytope} $\Gamma$ is defined by
  \[\Gamma=\{y\in\mR^{\binom{n}{2}}\mid \Trop(W_q)(y,1)\ge 0\}\]
Implicitly, we are ``tropicalizing" $q^i$ to $i$ by the evaluation $\Trop(W_q)(y,1)$. 
\end{defn}

\begin{ex}
  The superpotential polytope for $n=3$ corresponding to the potential \ref{potential-ex} is a polytope in $\mR^{\binom{3+1}{2}}$, with coordinates indexed by the $A_{ij}$ ordered lexicographically, defined by the inequalities:
  \begin{align*}
    A_{ij}\ge 0\\
    1-A_{11}-A_{12}-A_{13}\ge 0\\
    1-A_{11}-A_{12}-A_{23}\ge 0\\
    1-A_{11}-A_{22}-A_{23}\ge 0\\
    1-A_{11}-A_{22}-A_{33}\ge 0
  \end{align*}

\end{ex}

\begin{rmk}
  The choice of $1$ in the formula $\Trop(W_q)(x,1)\ge 0$ in the definition of $\Gamma$ corresponds to our choice of divisor $D=1*D_n$. In fact, \cite{RW} defines $\Delta_G(D)$ and $\Gamma_G(D)$ for more general divisors $D$ than just $D_n$ and more general seeds $G$. However, we have suppressed the dependence of $\Gamma$ on $G$ and $D$ because we have not discussed the cluster structure for $\mX^\vee$. This will be part of upcoming work \cite{upcoming}. 
\end{rmk}

\subsection{Poset polytope combinatorics}

In \cite{poset-poly}, Stanley associated two polytopes to a poset $P$: the \emph{order polytope} and the \emph{chain polytope}. The chain polytope lives in $\mR^{\abs{P}}$, and is defined by the inequalities $e_b \ge 0$ for any $b\in P$ and for any chain $b_1<b_2<\dots<b_k$ of $P$, we have $e_{b_1}+e_{b_2}+\dots+e_{b_k}\le 1$. In particular, because of the positivity inequalities, it is enough to consider the chain inequalities $e_{b_1}+e_{b_2}+\dots+e_{b_k}\le 1$ when $b_1<b_2<\dots<b_k$ is any maximal chain of $P$.

Let $\cP_n$ be the poset on the elements $\{b_{ij}\mid 1\le i\le j\le n\}$, with the cover relations $b_{ij}> b_{i+1j+1},b_{ij+1}$. The superpotential polytope $\Gamma$ produced above is the chain polytope of $\cP_n$: the terms $a_{ij}$ correspond to the positivity inequalities, and the terms $\frac{q}{\prod_{j\in [n]}a_{i_jj}}$ correspond to maximal chain inequalities. 

\begin{ex} 
  \begin{figure}[h!]
    \centering
    \begin{tikzpicture}
      \node (11) at (0,0) {$b_{11}$};
      \node (12) at (1,-1) {$b_{12}$};
      \node (13) at (2,-2) {$b_{13}$};
      \node (22) at (-1,-1) {$b_{22}$};
      \node (23) at (0,-2) {$b_{23}$};
      \node (33) at (-2,-2) {$b_{33}$};
      
      \draw (11)--(12)--(13);
      \draw (11)--(22)--(33);
      \draw (22)--(23);
      \draw (12)--(23);
    \end{tikzpicture}
    \caption{Hasse diagram of $\cP_3$}
  \end{figure}

  The four maximal chains of $\cP_3$ are $b_{33}\le b_{22}\le b_{11}$, $b_{23}\le b_{22}\le b_{11}$, $b_{23}\le b_{12}\le b_{11}$, and $b_{13}\le b_{12}\le b_{11}$. These correspond exactly to the `$q$' terms of the superpotential above, so the chain polytope of $\cP_3$ coincides with the superpotential polytope.  
\end{ex}

By \cite[Thm. 2.2 and Cor. 4.2]{poset-poly}, the superpotential polytope has as many vertices as antichains in $\cP_n$, and volume equal to the number of linear extensions of $\cP_n$.

\begin{lem} The number of antichains of $\cP_n$ is $C_{n+1}=\frac{1}{n+2} \binom{2n+2}{n+1}$.
\end{lem}

\begin{proof}
  Antichains are in bijection with order ideals, and for this poset the set of order ideals is in bijection with the set of Dyck paths of length $2n+2$ (Draw the Hasse diagram of $\cP_n$, and represent each poset element by a box, so that the resulting picture is of a tilted staircase. Then the desired bijection is obtained by associating to an antichain the Dyck path which passes over only those boxes in the corresponding order ideal.) The number of such Dyck paths is $C_{n+1}$.
\end{proof}

Hence the superpotential polytope has $C_{n+1}$ many vertices. This is also the number of Young diagrams contained in the $n\times n$ square up to transpose, and hence the number of distinct Pl\"ucker coordinates for the Lagrangian Grassmannian $\mX$. This bijection is described in more detail later as part of the proof that the superpotential polytope coincides with the Newton-Okounkov body.

\begin{ex}
  We illustrate the bijection between Dyck paths and antichains. 
  \begin{figure}[h!]
    \centering
    \begin{tikzpicture}
      \node (0) at (0,0) {$\cdot$};
      \node (1) at (-1,-1) {$\cdot$};
      \node (2) at (1,-1) {$\cdot$};
      \node (3) at (-2,-2) {$\cdot$};
      \node (4) at (0,-2) {$\cdot$};
      \node (5) at (2,-2) {$\cdot$};
      \node (6) at (-3,-3) {$\cdot$};
      \node (7) at (-1,-3) {$\cdot$};
      \node (8) at (1,-3) {$\cdot$};
      \node (9) at (3,-3) {$\cdot$};
      \node (10) at (-4,-4) {$\cdot$};
      \node (11) at (-2,-4) {$\cdot$};
      \node (12) at (0,-4) {$\cdot$};
      \node (13) at (2,-4) {$\cdot$};
      \node (14) at (4,-4) {$\cdot$};

      \node (a11) at (0,-1) {$b_{11}$};
      \node (a12) at (1,-2) {\color{mypurp}$b_{12}$};
      \node (a13) at (2,-3) {\color{mypurp}$b_{13}$};
      \node (a22) at (-1,-2) {$b_{22}$};
      \node (a23) at (0,-3) {\color{mypurp}$b_{23}$};
      \node (a33) at (-2,-3) {$b_{33}$};
      
      \draw (0)--(2);
      \draw (4)--(1)--(0);
      \draw (7)--(3)--(1);
      \draw (6)--(3);
      \draw (4)--(8)--(5);
      \draw (8)--(13)--(9);
      \draw (7)--(12)--(8);
      \draw[dashed,mypurp] (10)--(6);
      \draw[dashed,mypurp] (9)--(14);
      \draw[dashed,mypurp] (6)--(11)--(7)--(4)--(2)--(5)--(9);
    \end{tikzpicture}
    \caption{Dyck path (dashed, purple) corresponding to the antichain $\{a_{12}\}$, with corresponding order filter marked in purple.} 
  \end{figure}
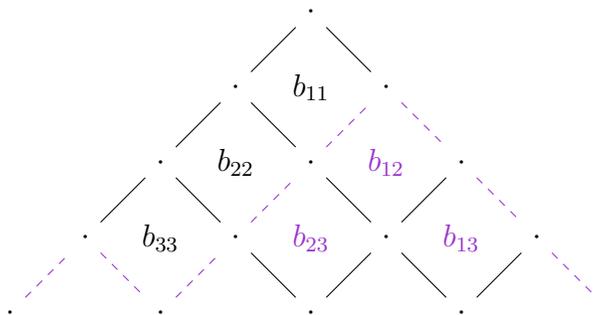
\end{ex}

\begin{lem}\label{super-volume} The number of linear extensions of $\cP_n$ is the degree of the Lagrangian Grassmannian $\mX$.
\end{lem}

\begin{proof}
  For any linear extension $L$ of a poset $P$, we can consider $L$ itself as a poset, and take its dual $L'$. Then $L'$ is a linear extension of the dual poset $P'$, for if $x\le y$ in $P'$, then $x\ge y$ in $P$, and hence $x\ge y$ in $L$, so that $x\le y$ in $L'$. Hence it suffices to count the number of linear extensions of the dual of $\cP_n$. Such a linear extension is equivalent to assigning a distinct integer $c_x$ in $[\binom{n+1}{2}]$ to each element $x$ of $\cP_n$, such that if $x\ge y$ in $\cP_n$, then $c_x\le c_y$, with equality iff $x=y$. This is exactly a standard Young tableaux on the staircase Young diagram $(n,n-1,\dots, 1)$. Hence the number of linear extensions of $\cP_n$ is the number of standard Young tableaux on the staircase diagram $(n,n-1,\dots, 1)$. The number of these is the degree of $\mX$. (See A005118.) 
\end{proof}

Hence the superpotential polytope has volume equal to $\deg(\mX)$. By our choice of $D=D_n$ (sections of $\cO(rD)$ correspond to degree $r$ homogeneous polynomials) and $G$ (the valuation is full rank, for example from \ref{unim}), the volume of the Newton-Okounkov bodies constructed above should also be equal to $\deg(\mX)$ (see e.g. \cite[Cor. 3.2]{kk} or \cite{km}, noting that we are not using the normalized volume, so should disregard the normalizing factor of $\dim(\mX)!$). In particular, it is reasonable to expect that the superpotential polytopes and the Newton-Okounkov bodies constructed in this section should be unimodularly equivalent, and we will prove this in the following section. 

\section{$\Delta_{\text{co-rect}}\cong \Gamma$}
Now we show that for the seed $G=G^{\text{co-rect}}_n$ and corresponding valuation $\val_{\text{co-rect}}$, the Newton-Okounkov body $\Delta_{\text{co-rect}}$ and superpotential polytope $\Gamma$ defined above are unimodularly equivalent, i.e. that there is a lattice isomorphism sending one polytope to the other. Our strategy is as follows. First, we define a linear map $M_n:\mR^{\binom{n+1}{2}}\rightarrow \mR^{\binom{n+1}{2}}$ with integer entries. We show next that $M_n$ is unimodular, and finally that $M_R(\Gamma)=\Delta_{\text{co-rect}}$. 

Recall first the superpotential polytope $\Gamma\subset \mR^{\binom{n+1}{2}}$ with coordinates $(A_{ij})$ ordered lexicographically by the indices and the Newton-Okounkov body $\Delta_{\text{co-rect}} \subset \mR^{\binom{n+1}{2}}$ with coordinates $(p_\lambda)$, where $\lambda$ is the complement of a rectangle, ordered reverse lexicographically by the indices. 

Consider the Pl\"ucker coordinates $p_\lambda$ such that $(n-1)\times(n-1)\subseteq \lambda\subsetneq n\times n$ and there are at least as many boxes to the right of the main diagonal (upper left to lower right) as there are below. Note that there are $\binom{n+1}{2}$ such Pl\"ucker coordinates: for each pair $(0\le i\le j\le n-1)$, associate the Young diagram containing the $(n-1)\times (n-1)$ rectangle with $j$ additional boxes to the right of the diagonal, and $i$ additional boxes below the diagonal.

Form the $\binom{n+1}{2}\times \binom{n+1}{2}$ matrix $M_n$ whose columns are the valuations of the above Pl\"ucker coordinates, ordered as follows: order first by \emph{decreasing} order (i.e. $2n\le 2n-1\le\cdots\le 1$) in the last entry, then break ties by \emph{increasing} order in the first to last, second to last, etc. entries. On the level of Young diagrams, this corresponds to ordering first by increasing number of additional boxes below the main diagonal, and then by dexreasing number of additional boxes right of the main diagonal. When these diagrams are indexed by pairs $(i,j)$, the ordering is $(i,j)\le (i',j')$ if $i<i'$ or $i=i'$ and $j\ge j'$.

\begin{ex}
 For $n=3$, the Pl\"ucker coordinates (in order) indexing the columns of $M_3$ are: $(p_{126},p_{136},p_{236},p_{125},p_{135},p_{124})$ or Young diagrams, $(p_{\tiny\yng(3,3)},p_{\tiny\yng(3,2)},p_{\tiny\yng(2,2)},p_{\tiny\yng(3,3,1)},p_{\tiny\yng(3,2,1)},p_{\tiny\yng(3,3,2)})$, or pairs $((0,2),(0,1),(0,0),(1,2),(1,1),(2,2))$. The Pl\"ucker coordinates indexing the rows are: $(p_{156},p_{126},p_{145},p_{125},p_{124},p_{123})$ or in Young diagrams, $(p_{\tiny\yng(3)},p_{\tiny\yng(3,3)},p_{\tiny\yng(3,1,1)},p_{\tiny\yng(3,3,1)},p_{\tiny\yng(3,3,2)},p_{\tiny\yng(3,3,3)})$. $M_3$ is the following matrix:
  \[M_3=
    \left(
      \begin{array}{cccccc}
        1&1&2&0&0&0\\ 
        1&2&2&1&2&0\\ 
        1&1&1&0&0&0\\ 
        2&2&2&1&2&0\\ 
        1&1&1&1&1&0\\ 
        1&1&1&1&1&1
      \end{array}
    \right)
  \]
  This matrix is unimodular.
\end{ex}

\subsection{Unimodularity}
\begin{lem}\label{ul}
  The upper left $n\times n$ block of $M_n$ has the form:
  \[
    \left(
      \begin{array}{cccccc}
        1 & 1 & \cdots & 1 & 1 & 2\\
        1 & 1 & \cdots & 1 & 2 & 2\\
        \vdots & & \ddots & & & \vdots\\
        1 & 2 & \cdots & 2 & 2 & 2\\
        1 & 1 & \cdots & 1 & 1 & 1          
      \end{array}
    \right)=
    \begin{cases} 1 & i=n\\
      1 & j\le n-i, i\neq n\\
      2 & j> n-i, i\neq n
    \end{cases}
  \]  
\end{lem}

\begin{proof} By our choice of ordering for Pl\"ucker coordinates indexing the columns, the $k^{th}$ column ($1\le k\le n$) is the valuation of the Pl\"ucker coordinate $p_{I_k}$, where $I_k=[n]\backslash(n-k+1) \cup \{2n\}$, corresponding to the pair $(0,n-k)$ or the Young diagram $\lambda_k=(n^{n-k},(n-1)^{k-1})$. The $k^{th}$ row corresponds to the $k\times n$ rectangle for $1\le k\le n-1$, and the hook $(n,1^{n-1})$ for $k=n$, and we will call this diagram $\mu_k$. Using the maxdiag formula:
  \[\val_{\text{co-rect}}(p_{\lambda_j})_{\mu_i}=
    \begin{cases}
      \mathrm{maxdiag}((n,1^{n-1})\backslash \lambda_j)=1 & i=n\\
      \mathrm{maxdiag}((i^n)\backslash \lambda_j) + \mathrm{maxdiag}((n^i)\backslash \lambda_j) & i\neq n\\
    \end{cases}    
  \]
  Let $i\neq n$. Since $(n-1)\times (n-1)\subset \lambda_j$, and $\lambda_j$ has only $n-1$ parts, then $\mathrm{maxdiag}((i^n)\backslash \lambda_j)=1$. On the other hand,
  \[(n^i)\backslash \lambda_j=
    \begin{cases}
      \varnothing & n-j\ge i\\
      (1^{i-(n-j)}) & n-j< i
    \end{cases}
  \]
  In particular,
  \[\val_{\text{co-rect}}(p_{\lambda_j})_{\mu_i} =
    \begin{cases} 1 & i=n\\
      1+\mathrm{maxdiag}((n^i)\backslash \lambda_j) & i\neq n\\
    \end{cases}=
    \begin{cases} 1 & i=n\\
      1+0=1 & n-j\ge i, i\neq n\\
      1=1=2 & n-j<i, i\neq n
    \end{cases}
  \]
\end{proof}

Applying the matrix with diagonal entries all $1$, super-diagonal entries $-1$, and all other entries $0$, followed by the matrix with diagonal entries all $1$, all non-diagonal entries in the first column $-1$, and all other entries $0$, both on the right, sends the upper left $n\times n$ block of $M_n$ to a permutation matrix (corresponding to the longest word). Each of these is a unimodular transformation, so the upper left $n\times n$ block of $M_n$ is unimodularly equivalent to $\Id_n$. We denote this sequence of unimodular transformations by $T_n$. In particular, the upper left block of $M_n$ is unimodular. 

\begin{lem}\label{lr}
  The lower right $\binom{n}{2}\times\binom{n}{2}$ block of $M_n$ is $M_{n-1}$. 
\end{lem}

\begin{proof}
  By our choice of ordering for the Pl\"ucker coordinates indexing the columns, the corresponding Young diagrams (indexing the columns in this block) all contain the hook partition $(n,1^{n-1})$ in the upper left. Similarly, by our choice of coordinates for the valuations, the corresponding Young diagrams (indexing the rows in this block) also all contain the hook $(n,1^{n-1})$ in the upper left.

  Hence, to compute the valuations $\val_{\text{co-rect}}(p_\lambda)_\mu$ (e.g. using the maxdiag formula) in this block, we can remove all the upper left hooks $(n,1^{n-1})$ from the corresponding $\lambda$ and $\mu$. The resulting matrix is exactly a copy of $M_{n-1}$.
\end{proof}

Furthermore, since all of the Pl\"ucker coordinates indexing the last $\binom{n}{2}$ columns of $M_n$ contain the hook $(n,1^{n-1})$, then in particular $\val_{\text{co-rect}}(p_{I_j})_{(n,1^{n-1})}=0$ for any $I_j$ in these columns. In other words, the bottom row of the upper right $n\times\binom{n}{2}$ submatrix of $M_n$ is the $0$ vector.

\begin{lem}\label{ll}
  The lower left $\binom{n}{2}\times n$ block of $M_n$ has all columns equal and nonzero. 
\end{lem}

\begin{proof}
  By choice of ordering on the valuation coordinates, the coordinates $\mu_i$ indexing these rows are all complements of rectangles $i\times j$ with both $i,j< n$. Since the Pl\"ucker coordinates $p_{\lambda_j}$ indexing the columns all contain $(n-1)\times (n-1)$ and have only $n-1$ parts, we always have $\mathrm{maxdiag}(\mu_i\backslash \lambda_j)=1$ independent of the Pl\"ucker coordinate $p_{\lambda_j}$ indexing the column. In particular, when we compute the valuation, we get
  \[\val_{\text{co-rect}}(p_{\lambda_j})_{\mu_i} =
    \begin{cases}
      \mathrm{maxdiag}(\mu_i\backslash \lambda_j) = 1 & \mu_i = \mu_i^T\\
      \mathrm{maxdiag}(\mu_i\backslash \lambda_j) + \mathrm{maxdiag}(\mu_i^T\backslash \lambda_j)= 1 + 1 = 2 & \mu_i\neq \mu_i^T
    \end{cases}
  \]
  independently of $j$. Hence all columns of this block are equal. 
\end{proof}

Note in particular that the $T_n$ described above applied on the right to the lower left $\binom{n}{2}\times n$ block of $M_n$ only gives nonzero entries in the last column, and furthermore the entries are given by $1$ if the row index is invariant under transpose and $2$ otherwise. 

\begin{prop}\label{unim}
  $M_n$ is unimodular. 
\end{prop} 

\begin{proof}
  We proceed by induction, proving a slightly stronger statement: $M_n$ is unimodularly equivalent by only column operations to a lower triangular matrix with $1$'s on the main diagonal. When $n=1$, $M_n$ is a $\binom{2}{2}\times\binom{2}{2}=1\times 1$ matrix. The only Pl\"ucker coordinate we consider is $p_{\varnothing}=p_2$. The valuation is in the coordinate $p_{\tiny\yng(1)}=p_1$. By either the flow model or the max-diag formula, we see that $\val(p_\varnothing)_{\tiny\yng(1)}=1$. Hence $M_1=1$.

  Now let $n>1$. By induction, we have some unimodular transformation $T$ consisting of only column operations such that $M_{n-1}T$ is lower triangular with $1$'s on the main diagonal. Now, apply the block diagonal (hence unimodular, since both of the blocks are unimodular) transformation with $T_n$, which only used column operations, in the upper left $n\times n$ block and $T$ in the lower right $\binom{n}{2}\times\binom{n}{2}$ block to $M_n$ on the right to get $M'=M_n\mathrm{diag}(T_n,T)$. 

  The upper left $n\times n$ block of $M'$ is $\Id_n$ by the discussion immediately following \ref{ul}. The lower left $\binom{n}{2}\times n$ block of $M'$ is nonzero in only the last column and zero everywhere else by the discussion immediately following \ref{ll}. Finally, the lower right $\binom{n}{2}\times \binom{n}{2}$ block of $M'$ is lower triangular with $1$'s on the main diagonal by induction and \ref{lr}. Furthermore, since we have only used column operations on the first $n$ columns and the last $\binom{n}{2}$ columns independently, the $0$ bottom row of the upper right $n\times\binom{n}{2}$ submatrix (see the discussion immediately following \ref{lr}) is preserved. Hence, we can perform further column operations to zero out the upper right $n\times \binom{n}{2}$  submatrix without affecting the lower right $\binom{n}{2}\times \binom{n}{2}$ submatrix. Composing these gives us the desired unimodular transformation. 
\end{proof}

\begin{ex} $M_3$ and $M_2$ are:
  \[M_3=
    \left(
      \begin{array}{cccccc}
        1&1&2&0&0&0\\ 
        1&2&2&1&2&0\\ 
        1&1&1&0&0&0\\ 
        2&2&2&1&2&0\\ 
        1&1&1&1&1&0\\ 
        1&1&1&1&1&1
      \end{array}
    \right),\quad M_2=
    \left(
      \begin{array}{ccc}
        1&2&0\\
        1&1&0\\
        1&1&1
      \end{array}
    \right)
  \]

  Note first that the lower right $\binom{3}{2}\times \binom{3}{2}$ matrix of $M_3$ is $M_2$, the lower left $\binom{3}{2}\times 3$ matrix has rank $1$, and the bottom row of the upper right $3\times \binom{3}{2}$ matrix is $0$. The verifications of these facts exactly follow the proofs of the lemmas. Analyzing the upper left $3\times 3$ matrix (the arrows indicate applying matrices on the right)
  \[
    \left(
      \begin{array}{ccc}
        1&1&2\\
        1&2&2\\
        1&1&1
      \end{array}
    \right)\xrightarrow{A}
    \left(
      \begin{array}{ccc}
        1&0&1\\
        1&1&0\\
        1&0&0
      \end{array}
    \right)\xrightarrow{B}
    \left(
      \begin{array}{ccc}
        0&0&1\\
        0&1&0\\
        1&0&0
      \end{array}
    \right)\xrightarrow{C}
    \left(
      \begin{array}{ccc}
        1&0&0\\
        0&1&0\\
        0&0&1
      \end{array}
    \right)
  \]
  where the matrices $A,B,C$ are given by
  \[A=
    \left(
      \begin{array}{ccc}
        1&-1&0\\
        0&1&-1\\
        0&0&1
      \end{array}
    \right),
    B=
    \left(
      \begin{array}{ccc}
        1&0&0\\
        -1&1&0\\
        -1&0&1
      \end{array}
    \right),
    C=
    \left(
      \begin{array}{ccc}
        0&0&1\\
        0&1&0\\
        1&0&0
      \end{array}
    \right)
  \]
  and their product is:
  \[ABC=
    \left(
      \begin{array}{ccc}
        0&-1&2\\
        -1&1&0\\
        1&0&-1
      \end{array}
    \right)
  \]

  The matrix sending $M_2$ to the desired form constructed in the proposition is:
  \[
    \left(
      \begin{array}{ccc}
        1&2&0\\
        1&1&0\\
        1&1&1
      \end{array}
    \right)
    \left(
      \begin{array}{ccc}
        -1&2&0\\
        1&-1&0\\
        0&0&1
      \end{array}
    \right)=
    \left(
      \begin{array}{ccc}
        1&0&0\\
        0&1&0\\
        0&1&1
      \end{array}
    \right)
  \]

  Putting these together, we apply first the following transformation to $M_3$:
  \[\left(
      \begin{array}{cccccc}
        1&1&2&0&0&0\\ 
        1&2&2&1&2&0\\ 
        1&1&1&0&0&0\\ 
        2&2&2&1&2&0\\ 
        1&1&1&1&1&0\\ 
        1&1&1&1&1&1
      \end{array}
    \right)
    \left(
      \begin{array}{cccccc}
        0&-1&2&0&0&0\\ 
        -1&1&0&0&0&0\\ 
        1&0&-1&0&0&0\\ 
        0&0&0&-1&2&0\\ 
        0&0&0&1&-1&0\\ 
        0&0&0&0&0&1
      \end{array}
    \right)=
    \left(
      \begin{array}{cccccc}
        1&0&0&0&0&0\\ 
        0&1&0&1&0&0\\ 
        0&0&1&0&0&0\\ 
        0&0&2&1&0&0\\ 
        0&0&1&0&1&0\\ 
        0&0&1&0&1&1
      \end{array}
    \right)
  \]

and then we finish by zeroing out the remaining nonzero entry above the diagonal using the second column. Hence $M_3$ is unimodular. 
\end{ex}

\subsection{Surjectivity}
It remains to prove that $M_n(\Gamma)=\Delta_{\text{co-rect}}$. To aid in our proof, we define an auxiliary polytope $\delta_G$ in the same ambient space as $\Delta_{\text{co-rect}}$:
\begin{defn}
  $\delta_G = \conv(\{\val_{\text{co-rect}}(p_\lambda)\mid \lambda\in \binom{[2n]}{n}\})$.
\end{defn}

A priori $\delta_G\subset \Delta_{\text{co-rect}}$ (for our choice of $G$ and $D$, $\delta_G = \conv(\val_{\text{co-rect}}(H^0(\mX,\cO(D))))$ is the $r=1$ part of the Newton-Okounkov body). For general $G$, $\delta_G\subsetneq \Delta_{\text{co-rect}}$ (i.e. the Pl\"ucker coordinates may not form a Khovanskii basis), but in our case, we will show that $\delta_G=\Delta_{\text{co-rect}}$ by showing that $M_n(\Gamma)=\delta_G$, and computing volumes.

Recall that the vertices of $\Gamma$ are characteristic functions of antichains of $\cP_n$. For each singleton antichain $\{a_{ij}\}$ of $\cP_n$, we associate the vertex $v_{ij}\in \Gamma$, and the hook partition $\nu_{ij}=(n+1-i,1^{j-i})$. Because $j\le n$, then $n+1-i>j-i$. 

\begin{lem}\label{hook-poset-bij}
The map $\{a_{ij}\}\rightarrow \nu_{ij}$ between singleton antichains of $\cP_n$ and nonempty hook partitions $(a,1^b)\subset n\times n$ with $a>b$ described above is a bijection.
\end{lem}

\begin{proof}
  For the hook partition $(a,1^b)\subset n\times n$ where $a>b$, set $i=n+1-a$ and $j=n+1-a+b$. We have $i\le j$ and $1\le i\le j\le n$ as required. Conversely, if $\nu_{ij}=\nu_{kl}$, then $i=k$ and $j=l$. Hence the above map is a bijection. 
\end{proof}
  
We first identify where these vertices are sent under $M_n$, and then use that to identify which antichains correspond to which Pl\"ucker coordinate valuations. 

\begin{lem}\label{singleton-antichain}
  $M_n(v_{ij})=\val_{\text{co-rect}}(p_\lambda)$, where $\lambda$ is the complement of $\nu_{ij}$ in the $n\times n$ square, and the complement is taken by right-justifying $\nu_{ij}$ in the bottom right corner. 
\end{lem}

\begin{proof}
  Since the vertices corresponding to the singleton antichains are unit vectors in $\mR^{\binom{n+1}{2}}$, their images under $M_n$ are the corresponding columns of $M_n$. In particular, since the columns of $M_n$ are totally ordered, $M_n$ provides an order-preserving bijection between the singleton antichains (lexicographic ordering) and the column labels. We will consider the singleton antichains as labelled by $a_{ij}$, and the columns as labelled by pairs $i,j$ with $0\le i,j\le n-1$. 

  Since there is a unique, order-preserving bijection between two totally ordered sets of the same finite cardinality, the action of $M_n$ can be reduced to finding such a bijection. The map $a_{ij}\mapsto (i-1,n-1-(j-i))$ is a bijection from the singleton antichains to the pairs labelling the columns of $M_n$. It remains to check that this is order preserving. Suppose $a_{ij}\le a_{i'j'}$, so that $i< i'$ or $i=i'$ and $j\le j'$. In the first case $i-1< i-1'$. In the case of equality, $i-1=i-1'$, and $n-1-(j-i)\ge n-1-(j'-i')$, so the map is order preserving.

  Hence $M_n$ must be the map sending $v_{ij}$ to the valuation of the Pl\"ucker coordinate indexed by the pair $(i-1, n-1-(j-i))$. Then this diagram is the complement of the hook $(n-(i-1),n-1-(n-1-(j-i)))=(n-i+1,1^{j-i})=\nu_{ij}$ (right justified, in the bottom right corner) as described. 
\end{proof}

Any partition $\lambda\subset n\times n$ has a right-justified complement partition $\lambda^c\subset n\times n$. We decompose $\lambda^c$ into a union of $k$ nonempty (right-justified) hooks $\lambda^c = \nu_1+\cdots  + \nu_k$, where the decomposition comes from taking the hooks from the boxes of $\lambda^c$ along the main diagonal.

\begin{ex}
  For ${\tiny\yng(1)}\subset {\tiny\yng(3,3,3)}$, the complement is ${\tiny\young(:\,\,,\,\,\,,\,\,\,)}$, which decomposes into the hooks ${\tiny\young(::\,,::\,,\,\,\,)}+{\tiny\young(:\,,\,\,)}$. For an asymmetric example, take ${\tiny\yng(2)}\subset {\tiny\yng(3,3,3)}$. The complement is ${\tiny\young(::\,,\,\,\,,\,\,\,)}$, and decomposes into the hooks ${\tiny\yng(2)}+{\tiny\young(::\,,::\,,\,\,\,)}$. 
\end{ex}

\begin{lem}\label{hook-antichain}
  Let $\lambda\subset n\times n$ be a partition with at least as many boxes above the main diagonal as below. Then the hook decomposition of the transpose of the complement corresponds to an antichain of $\cP_n$ under the bijection \ref{hook-poset-bij}. 
\end{lem}

\begin{proof}
  Let $\lambda^c=\nu_1+\cdots+\nu_k$ be the hook decomposition, and set the notation $\nu_i = (a_i,1^{b_i})$ for the hooks. By our assumption that $\lambda$ has at least as many boxes above the diagonal as below, we have that $a_i>b_i$. Then $\nu_i$ corresponds to the poset element $x_i=a_{n+1-a_i,b_i+n+1-a_i}$ (because $b_i+n+1-a_i=n+1-(a_i-b_i)\le n$, this is actually an element of $\cP_n$). Now suppose that there are $i,i'$ such that the $x_i\le x_{i'}$. By definition of $\cP_n$, this means:
  \[b_i+n+1-a_i \ge b_{i'}+n+1-a_{i'} \iff b_i-a_i\ge b_{i'}-a_{i'}\]
  \[0\le n+1-a_i-(n+1-a_{i'})\le b_i+n+1-a_i-(b_{i'}+n+1-a_{i'})\]
  \[\iff 0\le a_{i'}-a_i\le b_i-a_i-b_{i'}+a_{i'}\iff 0\le b_i-b_{i'}\text{ and } 0\le a_{i'}-a_i\] 

In order to come from a partition, the hooks $\nu_i$ must satisfy a nesting condition: for any $i$, we must have $a_{i+1}<a_i$ and $b_{i+1}<b_i$. Hence, for any $i<j$, we must have $a_j<a_i$ and $b_j<b_i$. However, these conditions contradict the ones above, so it could not have been that $x_i\le x_{i'}$. Hence for all $i,i'$, the pair $x_i,x_{i'}$ must be incomparable, so the set $\{x_i\}$ is an antichain of $\cP_n$. 
\end{proof}

\begin{ex}
   For ${\tiny\yng(1)}\subset {\tiny\yng(3,3,3)}$, the complement is ${\tiny\young(:\,\,,\,\,\,,\,\,\,)}$, which decomposes into the hooks ${\tiny\young(::\,,::\,,\,\,\,)}+{\tiny\young(:\,,\,\,)}$, corresponding to the elements $a_{13}$ and $a_{23}$, respectively. For ${\tiny\yng(2)}\subset {\tiny\yng(3,3,3)}$, the complement is ${\tiny\young(::\,,\,\,\,,\,\,\,)}$, and decomposes into the hooks ${\tiny\yng(2)}+{\tiny\young(::\,,::\,,\,\,\,)}$, corresponding to the elements $a_{13},a_{22}$. 
\end{ex}

\begin{lem}\label{maxdiag-decomp}
  Let $\lambda\subset n\times n$ be any partition with at least as many boxes above the main diagonal as below. Let $\lambda^c=\nu_1+\cdots+\nu_k$ be the hook decomposition of the complement. Then $\mathrm{maxdiag}(\mu\backslash\lambda)=\sum_i \mathrm{maxdiag}(\mu\backslash\lambda_i)$, where $\lambda_i$ is the complement of $\nu_i$ (right-justified in the bottom right corner of $n \times n$), for any $\mu$ labelling a face of $G_{\text{co-rect}}$. 
\end{lem}

\begin{proof}
  Since $\mu$ is a face label of $G_{\text{co-rect}}$, it is the complement of a rectangle $\rho\subset n\times n$ in the bottom right corner. Then
  \[\mathrm{maxdiag}(\mu\backslash\lambda_i) =
    \begin{cases}
      1 & \nu_i\not\subset \rho\\
      0 & \nu_i\subset \rho
    \end{cases}
  \]

  Hence, the sum $S=\sum_i \mathrm{maxdiag}(\mu\backslash\lambda_i)$ is the number of $i$ such that $\nu_i\not\subset\rho$. Because of the nesting condition of the hooks $\{\nu_1,\dots, \nu_k\}$, there is some $0\le i\le k$ such that $\nu_j\subset\rho$ for $j>i$ and $\nu_j\not\subset\rho$ for $j\le i$. Hence $S=i$. Furthermore, for any box $b$ of $\nu_i$ not contained in $\rho$, then this guarantees the existence of a diagonal of $\mu\backslash \lambda$ of length $i$ ending at $b$. (Since $b\in \nu_i$ and $b\notin \rho$, then $b\in \mu$. Since $\nu_i$ is the $i^{th}$ hook of the complement of $\lambda$, then the $i$ boxes diagonally above and to the left of $b$ are also not in $\lambda$ while being in $\mu$.) Hence $\mathrm{maxdiag}(\mu\backslash \lambda)\ge \sum_i \mathrm{maxdiag}(\mu\backslash \lambda_i)$.

  Conversely, start with a maximal diagonal of $\mu\backslash \lambda$. Since $\mu$ is the complement of a rectangle, we can assume that the diagonal has its corner at one of the (at most) two corners of $\mu$. (Since $\lambda$ is a partition, if $b\in (\mu\backslash\lambda)$, then $b'\in(\mu\backslash\lambda)$ for any $b'$ below or to the right of $b$ in $\mu$. In particular, since $\mu$ is the complement of a rectangle, this means that we can always either move the diagonal down or to the right unless its corner aligns with a corner of $\mu$.) Since $\nu_j$ meets the $j^{th}$ box of the diagonal (counting from the bottom-most box), then the translate of $\nu_j$ to the bottom right meets the $1^{st}$ box of the diagonal, hence $\nu_j$ is not contained in $\rho$, so contributes $1$ to $S$. Hence $\mathrm{maxdiag}(\mu\backslash\lambda) \le \sum_i \mathrm{maxdiag}(\mu\backslash \lambda_i)$. 
\end{proof}

\begin{cor}\label{main-thm}
  Let $\lambda\subset n\times n$ be any partition with at least as many boxes above the main diagonal as below. Let $\lambda^c=\nu_1+\cdots+\nu_k$ be the hook decomposition with corresponding antichain $\{x_i\}$, using the bijection from \ref{hook-antichain}. Then $M_n$ sends the vertex of $\Gamma$ corresponding to the antichain $\{x_i\}_{i=1}^k$ to $\val_{\text{co-rect}}(p_\lambda)$. Hence $\Delta_{\text{co-rect}}\cong\Gamma$.
\end{cor}

\begin{proof}
  Applying \ref{maxdiag-decomp} coordinatewise, and keeping the same notation, we get the equation
  \[\val_{\text{co-rect}}(p_\lambda)=\sum_{i=1}^k \val_{\text{co-rect}}(p_{\lambda_i})\]
  By \ref{singleton-antichain} and \ref{hook-antichain}, the right hand side is exactly $M_n$ acting on the antichain $\{x_i\}_{i=1}^k$. Hence every valuation of a Pl\"ucker coordinate is obtained by $M_n$ acting on the vertices of $\Gamma$, and $M_n(\Gamma)=\delta_G$. Since $M_n$ is unimodular, then $\vol(\delta_G)=\vol(\Gamma)=\deg(\mX)$ by \ref{super-volume}. Since $\delta_G\subset \Delta_{\text{co-rect}}$ have the same volume and are both closed and convex, they must be equal. Hence $M_n(\Gamma)=\Delta_{\text{co-rect}}$. 
\end{proof}

\begin{ex} Adding the bottom two rows gives the top row, as desired. 
  \[
    \begin{array}{c|c|c|c|c|c|c}
      \mu & {\tiny\yng(3)} & {\tiny\yng(3,3)} & {\tiny\yng(3,1,1)} & {\tiny\yng(3,3,1)} & {\tiny\yng(3,3,2)} & {\tiny\yng(3,3,3)}\\ \hline
      \val_{\text{co-rect}}(p_{\tiny\yng(2)})_\mu & 2 & 3 & 1 & 3 & 2 & 2\\
      \val_{\text{co-rect}}(p_{\tiny\yng(3,3,1)})_\mu & 0 & 1 & 0 & 1 & 1 & 1\\
      \val_{\text{co-rect}}(p_{\tiny\yng(2,2)})_\mu & 2 & 2 & 1 & 2 & 1 & 1\\
      
    \end{array}
  \]
\end{ex}

\section{Future Work}

The main theorem in \cite{RW} about the two polytopes $\Delta_G$ and $\Gamma_G$ is that they are equal. In order to obtain this theorem in our situation, we need an explicit cluster structure on the type $B$ orthogonal Grassmannian $\mathrm{OG}(n,2n+1)$ to carry out the remainder of the proof strategy. In upcoming work \cite{upcoming}, we will present a cluster structure for the type $B$ orthogonal Grassmannians $\mathrm{OG}(n,2n+1)$. This will then allow us to use the cluster structure along with the result in this article to study cluster duality for $\mX$ and $\mX^\vee$. 

\bibliographystyle{amsalpha}
\bibliography{/home/bib.bib}

\end{document}